\documentclass[11pt]{article}

\usepackage{amsmath, amsfonts, amssymb}
\usepackage{mathrsfs}
\usepackage{theorem}
\usepackage{bm}

\newtheorem{thm}{Theorem}[section]
\newtheorem{prop}[thm]{Proposition}
\newtheorem{lem}[thm]{Lemma}

\newtheorem{defn}[thm]{Definition}

{\theorembodyfont{\rmfamily} \newtheorem{rem}[thm]{Remark}}
{\theorembodyfont{\rmfamily} \newtheorem{exam}{Example}[section]}

\newcommand{\ra}{\rightarrow}
\newcommand{\dis}{\displaystyle}

\def\R{\mathbb R}

\def\d{\text{\rm{d}}}
\def\E{\mathbb E}
\def\p{\mathbb P}\def\e{\text{\rm{e}}}

\def\La{\Lambda}
\def\veps{\varepsilon}
\def\diag{\mathrm{diag}}

\def\S{\mathcal S}

\newcommand{\fin}{\hspace*{\fill}\rule{0.3em}{1ex}}
\newenvironment{proof}{{\bf \noindent Proof.}}{\fin}

\numberwithin{equation}{section}

\allowdisplaybreaks

\begin{document}

\title{Algebraic stability of non-homogeneous regime-switching diffusion processes
\footnote{Supported in part by NSFC (No.11301030), NNSFC (No.11431014), 985-project.}
}

\author{Jing Li \qquad Jinghai Shao \footnote{ Email: shaojh@bnu.edu.cn}\\[0.5cm]
 School of Mathematical Sciences, Beijing Normal University, Beijing 100875, China}
\maketitle

\begin{abstract}
  Some sufficient conditions on the algebraic stability of non-homogeneous  regime-switching diffusion processes are established.
  In this work we focus on determining the decay rate of a stochastic system which switches randomly between different states, and owns different decay rates at various states. In particular,  we show that if a finite-state regime-switching diffusion process is $p$-th moment exponentially stable in some states and is $p$-th moment algebraically stable in other states, which are characterized by a common Lyapunov function, then this process is ultimately exponentially stable regardless of the jumping rate of the random switching between these states. Moreover, these results can be applied to the stabilization of SDE by feedback control to reduce the observation times.
\end{abstract}



\section{Introduction}
The purpose of this work is to study the algebraic stability of regime-switching diffusion processes. The stability of a stochastic process is an important property, and has many important applications in practice. There are extensive researches on the topic of exponential stability of stochastic processes, for example, \cite{AOP}, \cite{Kh}, \cite{Mao} among others. However, there are a large amount of stochastic processes which are stable but subject to a lower decay rate other than exponential decay. Consequently, it appears to be necessary to extend the usual results on exponential stability to algebraic stability, for instance, polynomial or logarithmic stability. To this aim, K. Liu and A.Y. Chen in \cite{LC} have provided some sufficient conditions for the solutions of stochastic differential equations with and without delay.   On the other hand, a wide variety of natural and engineering systems are inherently in a random environment in the sense that several dynamical models are required to describe their behavior when the environment is at different state. For example, concerning ecological population modeling, when two or more species live in proximity and share the same basic requirements, they usually compete for resources, food or territory. But the growth rates and the carrying capacities are often subject to environment noise. For instance, the growth rates of some species in the rainy season are much different from those in the dry season. In addition, the environment usually changes randomly. Owing to the advantage of providing more realistic models in applications, regime-switching diffusion processes, also known as hybrid diffusion processes, have received a fair amount of attention recently. See, for instance, \cite{BGM, Gh, MY, MYY, SX14, Sh15a, Sh15b, WSK, YX10, YZ} and references therein. However, due to the interaction of continuous dynamics and discrete events, the analysis of switching systems is more difficult and many new problems have been raised.

Let us consider the following simple regime-switching process $(X_t)_{t\geq 0}\in\R^d$ where $X_t$ solves $\d X_t= A_{\La_t}X_t\d t$ with $\La_t$ a Markov chain on a finite state space $\S=\{1,\ldots,N\}$ and $\{A_i\}_{i\in\S}$ a set of $d\times d$ real matrices. In \cite{BBMZ}, M. Benaim et al. showed that in the case $\S=\{1,2\}$, when $A_1$ and $A_2$ are both Hurwitz (all eigenvalues have strictly negative real part), $|X_t|$ may converge to $0$ or $\infty$ as $t\ra \infty$ depending on the switching rate as long as an average matrix $\bar A=\lambda A_1+(1-\lambda)A_2$ has a positive eigenvalue for some $\lambda\in (0,1)$. Furthermore, S. Lawley et al. \cite{Law} proved that the assumption that the average matrix has a positive eigenvalue is not necessary to ensure a blowup. These results show that the dynamics of the regime-switching processes can be very different from the individual dynamics in each fixed state and the averaged dynamics.

Precisely,  the regime-switching diffusion process considered in this work owns two components $(X_t, \La_t)$, where $(X_t)$ satisfies the following stochastic differential equation (SDE)
\begin{equation}\label{1.1}
\d X_t=b(t,X_t,\La_t )\d t+\sigma(t,X_t,\La_t)\d B_t,
\end{equation}
and $(\La_t)$ is a continuous time Markov chain on a finite state space $\S=\{1,2,\ldots,N\}$ with $Q$-matrix $(q_{ij})$. $(B_t)$ is a Brownian motion in $\R^d$ independent of $(\La_t)$. Throughout this work, we assume that $(q_{ij})$ is irreducible and conservative, which means, in particular, $q_k=\sum_{j\neq k} q_{kj}$ for each $k\in \S$. Here, the Markov chain $(\La_t)$ is used to describe the random switching of the environment between different states. It can be seen from equation \eqref{1.1} that the behavior of $(X_t)$
may be very different when the environment varies in different states, i.e. $\La_t$ takes different values. The randomness of $(\La_t)$ can cause essential difference to this system. For instance, there are examples (cf. \cite{PS}) to show that even when $(X_t)$ is ergodic in every fixed environment, $(X_t)$ itself could be transient when the environment changes randomly between these states. We refer the readers to the recent works \cite{CH, Sh15a, Sh15b} and references therein on the recurrent properties of regime-switching diffusion processes, where some criteria are provided to justify the recurrent property of regime-switching diffusion processes. This work is a continuation of our previous work \cite{SX14} on the stability of regime-switching diffusion processes. In \cite{SX14},  we focused on the time homogeneous regime-switching diffusion processes and the exponential stability in probability.  This work is devoted to studying the algebraic stability of non-homogeneous regime-switching diffusion processes.

Let us illustrate our motivation of this work via the following simple examples. Define the processes $(X_t)$, $(Y_t)$ and $(Z_t)$ by the following SDEs:
\begin{equation}\label{e-1}
\d X_t=-X_t\d t+X_t\d B_t, \quad X_0=x\in \R.
\end{equation}
\begin{equation}\label{e-2}
\d Y_t=-\frac{Y_t}{1+t} \d t+\frac{Y_t}{\sqrt{1+t}}\d B_t,\quad Y_0=y\in \R.
\end{equation}
\begin{equation}\label{e-3}
\d Z_t=-\frac{m Z_t}{1+t}\d t+\frac{\sqrt{m} Z_t}{\sqrt{1+t}}\d B_t,  \ \   Z_0=z\in\R,\ m>1.
\end{equation}
Then  we have
\begin{align*}&\limsup_{t\ra \infty} \frac{\log \E[X_t^2]}{t}\leq -1, \quad \limsup_{t\ra \infty}\frac{ \log \E [ Y_t^2]}{ \log( 1+t)}\leq -1, \ \
\limsup_{t\ra \infty}\frac{\log \E[ Z_t^2]}{ \log (1+t)^m}\leq -1.
\end{align*}
This means that $(X_t)$ is mean square stable with exponential decay; $(Y_t)$ is mean square stable  with polynomial decay of order 1; $(Z_t)$ is mean square stable with polynomial decay of order $m$.
\begin{exam}\label{ex-1}
Consider the following regime-switching diffusion process:
 \begin{equation}\label{e-4}
 \d X_t=\Big(-X_t\mathbf{1}_{\La_t=1}-\frac{X_t}{1+t}\mathbf{1}_{\La_t=2}\Big)\d t+\Big(X_t\mathbf{1}_{\La_t=1}+\frac{X_t}{\sqrt{1+t}}\mathbf{1}_{\La_t=2}\Big)\d B_t.
 \end{equation}
The process $(X_t)$ is mean square stable with exponential decay when the environment $\La_t$ is at the state ``1", and is mean square stable with polynomial decay of order 1  when $\La_t$ is at the state ``2".  The strong solution of this system exists uniquely (see, e.g. \cite{MY}, or \cite{Sh15c} for regime-switching diffusions with more general coefficients and state-dependent switching).  Some natural and important questions are whether the process $(X_t)$ is still mean square stable, and if so, what is the decay rate of the process $(X_t)$  when the environment $(\La_t)$ changes randomly between states ``1" and ``2" according to the $Q$-matrix $(q_{ij})$.
\end{exam}

\begin{exam}\label{ex-2}
Consider the following regime-switching diffusion process:
\begin{equation}\label{e-5}
\d X_t=-\Big(\frac{X_t}{1+t}\mathbf{1}_{\La_t=1}+\frac{2 X_t}{1+t}\mathbf{1}_{\La_t=2}\Big)\d t+\Big(\frac{X_t}{\sqrt{1+t}}\mathbf{1}_{\La_t=1}+\frac{\sqrt{2} X_t}{\sqrt{1+t}}\mathbf{1}_{\La_t=2}\Big)\d B_t.
\end{equation}
The process $(X_t)$ is mean square stable with polynomial decay when $(\La_t)$ takes the fixed value ``1" or ``2", but of different order.  What is the decay rate of $(X_t)$  when $(\La_t)$ changes randomly between states ``1" and ``2"?
\end{exam}

Although the process defined by \eqref{e-5} is very simple, to answer the question posed above is rather nontrivial. In fact,  direct calculation yields
\[\d \log X_t^2=-\frac{3\La_t}{1+t}\d t+\frac{2\sqrt{\La_t}}{\sqrt{1+t}}\d B_t.\]
Hence,
\begin{equation*}
  X_t^2=x_0^2\exp\Big(-\int_0^t\frac{3\La_s}{1+s}\d s+\int_0^t\frac{2\sqrt{\La_s}}{\sqrt{1+s}}\d B_s\Big),
\end{equation*}
and
\begin{equation}\label{1.7}
  \E X_t^2=x_0^2\,\E e^{-\int_0^t\frac{\La_s}{1+s}\d s}.
\end{equation}
Although the right hand side of \eqref{1.7} is simple, it is difficult to calculate or estimate its precise decay rate. Formula \eqref{1.7} reveals that the decay rate of a regime-switching diffusion process depends on the jumping times of the Markov chain $(\La_t)$ rather than simply on the sojourn time of $(\La_t)$ at each fixed state.  One can not expect to use the joint distribution of the jumping times of the Markov chain $(\La_t)$ to derive some useful estimation.

According to Theorem \ref{cor-1} below, the regime-switching diffusion process $(X_t)$ defined by \eqref{e-4} is mean square stable with exponential decay regardless of the  jumping rate $(q_{ij})$ of $(\La_t)$ between states ``1" and ``2". This phenomenon of independence of jumping rate $(q_{ij})$ is interesting from the theoretic viewpoint, and of importance in application to stabilize a switched system. This kind of phenomenon also appeared when we discussed the strong ergodicity of regime-switching diffusion processes in \cite{SX13}. An intuitive reason of this fact is that the exponential decay is fast enough that it can not be decelerated to polynomial decay by a process with  polynomial decay.
On the other hand, according to Theorem \ref{prop-2} below, the process $(X_t)$ satisfying \eqref{e-5} is  mean square stable with polynomial decay of order $\theta$, locating in the interval $(1,2)$ and depending on the jumping rates of $(\La_t)$.  This reflects the complexity of the stability  of regime-switching systems.

In Section 2 we first establish some sufficient conditions for the stability of regime-switching diffusion processes in terms of Lyapunov functions. We also provide the corresponding extension of results in \cite{SX14} for non-homogeneous regime-switching diffusion processes. Analogous to \cite{SX14}, these results can be generalized to study the stability of state-dependent regime-switching diffusion processes in an infinite countable state space. We do not present the results in this direction as it is not our main purpose of this work. Main results of this work are Theorems \ref{prop-2} and \ref{cor-1}. These results reflect the essential characteristics of the regime-switching diffusion processes. At last, we apply our results to the stabilization of a SDE  by a feedback control which can reduce greatly the observation times.

\section{Stability of regime-switching diffusion processes}
In this section we study the algebraic stability of regime-switching diffusion processes. Assume that SDE \eqref{1.1} admits a unique global solution throughout this work, and refer the reader to \cite{MY} or \cite{Sh15c} for concrete conditions on coefficients to ensure this assumption holds. Moreover, we refer to \cite[Chapter 5]{MY}, where X. Mao and C. Yuan have discussed many kinds of stability for regime-switching diffusion processes, including the asymptotical stability in $p$-th moment (Theorem 5.29), $p$-th moment exponential stability (Theorem 5.8), etc.

\begin{defn}\label{t-2-1}
Assume that $\lambda(t)>0$ and $\lambda(t)\uparrow \infty$ as $t\ra \infty$.
A stochastic process $(X_t)$ is said to be  $p$-th moment stable ($p>0$) with decay $\lambda(t)$ of order $\gamma>0$ if there exists a pair of positive constants $\gamma$ and $C(x_0)$ such that
$\dis \E|X_t|^p\leq C(x_0)\lambda(t)^{-\gamma}$, $t\geq 0$,
or equivalently,
\[\limsup_{t\ra \infty}\frac{\log \E|X_t|^p}{\log \lambda(t)}\leq -\gamma.\]
In particular, if $\lambda(t)=e^t$, then $(X_t)$ is called $p$-th moment stable with exponential decay. If $\lim_{t\ra \infty}\frac{\lambda(t)}{e^t}=0$, then $(X_t)$ is called $p$-th moment stable with algebraic decay.
\end{defn}

Let $\mathscr A_t $ be the infinitesimal generator of $(X_t,\La_t)$ satisfying SDE \eqref{1.1}, i.e.
\begin{align*}\mathscr A_t f(t,x,i)&=\frac{\partial f}{\partial t}(t,x,i) + \sum_{k=1}^d b_k(t,x,i)\frac{\partial f}{\partial x_k}(t,x,i) \\ &\quad+ \frac 12\sum_{k,l=1}^da_{kl}(t,x, i)\frac{\partial^2 f}{\partial x_k\partial x_l}(t,x,i)\!+\!\sum_{j\in \S}q_{ij}f(t,x,j)
      \end{align*}
      for any $f\in C^{1,2}(\R^+\!\times\R^d\!\times\S)$, where $(a_{kl})=\sigma\sigma^\ast$.
\begin{lem}\label{t-sw}
Let $(X_t,\La_t)$ be a regime-switching diffusion process satisfying SDE \eqref{1.1} and $(\La_t)$ be a Markov chain on $\S=\{1,\ldots,N\}$ with conservative and irreducible $Q$-matrix $(q_{ij})$. Assume that there exists a nonnegative function $V$ on $\R^+\!\times \!\R^d\times\!\S$ such that $x\mapsto V(t,x,i)$ is twice differentiable, $t\mapsto V(t,x,i)$ is differentiable, and
\begin{gather}
    \label{3.1}V(t,x,i)\geq |x|^p\quad \forall\, (t,x,i)\in \R^+\times \R^d\times \S;\\
    \label{3.2} \mathscr A_tV(t,x,i)\leq -\psi(t)V(t,x,i)\quad \forall\, (t,x,i)\in \R^+\!\times\R^d\!\times\S,
\end{gather} for some  function $ \psi\in C(\R^+;\R^+)$.
Set $\alpha(t)=\exp\big(\int_0^t\psi(s)\d s\big)$ for $t\geq 0$.
Suppose that $\lim_{t\ra \infty}\alpha(t)= \infty$, then $(X_t)$ is $p$-th moment stable with decay rate $\alpha(t)$ of order 1, i.e.
\begin{equation}\label{2.8}
\limsup_{t\ra\infty}\frac{\log \E |X_t|^p}{\log \alpha(t)}\leq -1.
\end{equation}
Moreover, it also holds that $(X_t)$ is almost surely asymptotically stable, i.e.
\[\lim_{t\ra \infty} X_t =0,\quad \mathrm{a.s.}\]
\end{lem}

\begin{proof} According to It\^o's formula for $(X_t,\La_t)$ (cf. \cite{Sk} or
  \cite[Theorem 1.45]{MY}), one gets
  \[\d V(t,X_t,\La_t)=\mathscr A_t V(t,X_t,\La_t)\d t+\d M_t,\quad \text{where $(M_t)$ is a martingale}.\]
  Then, for $0\leq s<t$,
  \begin{align*}
    \E V(t,X_t,\La_t)&=\E V(s,X_s,\La_s)+\int_s^t\E\mathscr A_r V(r,X_r,\La_r)\d r\\
    &\leq \E V(s,X_s,\La_s)-\int_s^t\psi(r)\E V(r,X_r,\La_r)\d r.
  \end{align*}
  Set $u(t)=\E V(t,X_t,\La_t)$ for $t\geq 0$. The previous inequality can be rewritten as \[ u(t)-u(s)\leq -\int_s^t\psi(r)u(r)\d r, \quad 0\leq s<t.\]
  Dividing both sides by $t-s$, then letting $s\uparrow t$, we obtain $u'(t)\leq -\psi(t) u(t)$. Applying Gronwall's inequality, we get
  \[\E V(t,X_t,\La_t)=u(t)\leq \E V(0,X_0,\La_0)\exp\Big(-\!\int_0^t\!\psi(r)\d r\Big).
  \]
  By virtue of $V(t,x,i)\geq |x|^p$, it follows that
  \[\E|X_t|^p\leq \E V(t,X_t,\La_t)\leq \E V(0,X_0,\La_0)\exp\Big(-\!\int_0^t\!\psi(r)\d r\Big).\]
  Consequently,
  \[\limsup_{t\ra \infty}\frac{\log \E |X_t|^p}{\log \alpha(t)} \leq -1,\]
  as $\lim_{t\ra \infty} \alpha(t)=\infty$.

  For the sake of asymptotical stability, we  shall use the convergence theorem of semi-martingales (cf. \cite[Theorem 7, pp. 139]{LS}).
  According to It\^o's formula, it holds
  \[\d \Big(V(t,X_t,\La_t)\e^{\int_0^t \psi(s)\d s}\Big)=\big(\psi(t)V(t,X_t,\La_t)+\mathscr A_t V(t,X_t,\La_t)\big)\e^{\int_0^t\psi(s)\d s}\d t+\d \widetilde M(t),\]
  where $(\widetilde M_t)$ is a local martingale with $\widetilde M(0)=0$. Due to \eqref{3.2}, we get
  \[V(t,X_t,\La_t)\e^{\int_0^t\psi(s)\d s}\leq V(0,X_0,\La_0)+\widetilde M(t),\]
  which yields that
  \begin{equation}\label{3.4}\lim_{t\ra \infty} V(t,X_t,\La_t)\e^{\int_0^t\psi(s)\d s}<\infty, \quad \mathrm{a.s.}
  \end{equation}
  Furthermore, as $\lim_{t\ra\infty} \int_0^t\psi(s)\d s=\infty$, it follows from \eqref{3.4} that $\lim_{t\ra\infty} V(t,X_t,\La_t)=0$ a.s. By using the control \eqref{3.1}, it is easy to see $\lim_{t\ra \infty} X_t=0$ a.s.
\end{proof}

In \cite{LC}, K. Liu and A.Y. Chen established the following criterion on the stability of stochastic processes. Consider
\begin{equation}\label{2.1}
\d X_t=b(t,X_t)\d t +\sigma(t,X_t)\d B_t,\ \  X_0=x_0\in \R^d,
\end{equation} where $(B_t)$ is a Brownian motion in $\R^d$.
The infinitesimal generator of $(X_t)$ is given by
\begin{equation}\label{2.2}L_t f(t,x)=\frac{\partial f}{\partial t}+\sum_{i=1}^db_i(t,x)\frac{\partial f}{\partial x_i}+\frac12\sum_{i,j=1}^d a_{ij}(t,x)\frac{\partial^2f}{\partial x_i\partial x_j},
\end{equation}
where $(a_{ij})=\sigma\sigma^\ast$, and $\sigma^\ast$ denotes the transpose of the matrix $(\sigma_{ij})$. Here, the subscript $t$ of $L_t$ is used to emphasize that the infinitesimal generator of $(X_t)$ is time-dependent.

\begin{lem}[\cite{LC}]\label{LC}
Let $\rho\in C^{1,2}(\R^+\times \R^d;\R^+)$ and $\varphi_1$, $\varphi_2$ be two continuous nonnegative functions. Let $\lambda(t)>0$ and $\lambda(t)\uparrow \infty$ as $t\ra \infty$. Assume that for all $x\in \R^d$, $t\geq 0$, there exist   constant  $p>0$  and real numbers $\nu$, $\theta$ such that
\begin{itemize}
  \item[(1)]\ $|x|^p\lambda(t) \leq \rho(t,x),\quad (t,x)\in \R^+\times \R^d$;
  \item[(2)]\ $L_t \rho(t,x)\leq \varphi_1(t)+\varphi_2(t)\rho(t,x),\quad (t,x)\in \R^+\times \R^d$;
  \item[(3)]\ $\displaystyle \limsup_{t\ra \infty}\frac{\log\Big(\int_0^t\varphi_1(s)\d s\Big)}{\log \lambda(t)}\leq \nu$,\ $\dis \limsup_{t\ra \infty}\frac{\int_0^t\varphi_2(s)\d s}{\log\lambda(t)}\leq \theta.$
\end{itemize}Then, whenever $\gamma:=1-\theta-\nu>0$, the solution $(X_t)$ of \eqref{2.1} is $p$-th moment stable with decay $\lambda(t)$ of order $\gamma$, that is,
  \[\limsup_{t\ra \infty}\frac{\log \E|X_t|^p}{\log \lambda(t)}\leq -\gamma<0.\]
\end{lem}
We mention that Lemma \ref{t-sw} can also be used to investigate the stability of the solutions of stochastic differential equations without regime-switching (i.e. $N=1$), which simplifies and improves the  criterion  Lemma \ref{LC} in some situations. Indeed, if $\int_0^\infty \e^{-\int_0^t\phi_2(s)\d s}\phi_1(t)\d t<\infty$ in Lemma \ref{LC}, and the function $\lambda$ is differentiable, we can find the functions $V$ and $\psi$ satisfying the conditions of Lemma \ref{t-sw}. In fact, take
\[V(t,x)=\lambda(t)^{-1}\rho(t,x)+\frac{\beta(t)}{\alpha(t)\lambda(t)},\]
where $\alpha(t)=\alpha_0\exp\big(-\int_0^t\phi_2(r)\d r\big)$, $\alpha_0$ is a positive constant, and $\beta(t)=\beta_0-\int_0^t\alpha(s)\phi_1(s)\d s$ with a constant $\beta_0>\int_0^\infty \alpha(t)\phi_1(t)\d t$. Then $\mathscr A_tV(t,x)\leq -\big(\frac{\alpha'(t)}{\alpha(t)}+\frac{\lambda'(t)}{\alpha(t)}\big)V(t,x)$. According to Lemma \ref{t-sw}, it  yields that
\[\limsup_{t\ra \infty} \frac{\log \E |X_t|^p}{\log \lambda(t)}\leq \theta-1,\]
which improves Lemma \ref{LC} a little bit since $\theta-1\leq \nu+\theta-1$ by noting that the constants $\nu$, $\theta$ in the condition (3) of Lemma \ref{LC} must be nonnegative.

Lemma \ref{t-sw} provides us a sufficient condition for algebraic stability in terms of Lyapunov functions.  But the coexistence of generators for diffusion processes and Markov chains increases greatly the difficulty to construct the desired Lyapunov functions for regime-switching diffusion processes.  So we need to find more concrete conditions to reflect the characterization of regime-switching processes so as to answer the questions posed in the introduction.

Now we state our first main result of this work, which can be used to answer the questions posed in Examples \ref{ex-1} and \ref{ex-2}.

Let us first introduce a notation.
Let $(X_t,\La_t)$ be defined by \eqref{1.1}. Set
\[L^{(i)}_t f=\frac{\partial f}{\partial t}+\sum_{k=1}^db_k(t,x,i)\frac{\partial f}{\partial x_k}+\frac 12\sum_{k,l=1}^d a_{kl}(t,x,i)\frac{\partial^2 f}{\partial x_k\partial x_l}, \ f\in C^{1,2}(\R^+\!\times \R^d), \]
for every $i\in\S$.
\begin{thm}\label{prop-2}
  Let $(X_t,\La_t)$ be a regime-switching diffusion process satisfying \eqref{1.1}, where $(\La_t)$ is a Markov chain on $\S=\{1,2\}$ with $Q$-matrix $(q_{kl})$. Let $\psi\in C^1(\R^+)$ be a positive decreasing function satisfying $\lim_{t\ra \infty} \psi(t)=0$ and $\lim_{t\ra \infty}\int_0^t\psi(s)\d s=\infty$. Assume that there exist a function $\rho\in C^{1,2}(\R^+\times\R^d;\R^+)$, constants $p>0$ and $\theta\in [0,1)$ such that $\rho(t,x)\geq |x|^p$,
  \begin{equation}\label{con-1}
  \begin{split}
    L^{(1)}_t\rho(t,x)&\leq -\theta\psi(t)\rho(t,x),\\
    L^{(2)}_t\rho(t,x)&\leq -\psi(t)\rho(t,x),
  \end{split}
  \end{equation} for all $(t,x)\in \R^+\times\R^d$.
  Set $\kappa=\frac{q_{21}}{q_{12}}=\frac{q_2}{q_1}$,   and $\eta(\theta,\kappa)=\frac{1+\kappa \theta}{1+\kappa}$.
  Then
  \begin{equation}\label{2.9}
  \limsup_{t\ra \infty} \frac{\log \E|X_t|^p}{\int_0^t\psi(s)\d s}\leq -\eta(\theta,\kappa)<0.
  \end{equation}
  Moreover, $(X_t)$ is almost surely asymptotically stable, i.e. $\lim_{t\ra \infty} X_t=0$ a.s.
\end{thm}

\begin{proof}
  We shall use Lemma \ref{t-sw} to prove this theorem, and the key point is to find a suitable Lyapunov function $V(t,x,i)$. To this aim, let us first introduce two auxiliary functions $\beta_1(t)$ and $\beta_2(t)$.
  Let $\beta_1(t)\in C^1(\R^+)$ be a positive function and $\beta_1(t)\geq c>0$ for $t\geq 0$.  Set
  \[h(t)=\frac{(\lambda-\theta) \psi(t)}{q_{12}},\quad \beta_2(t)=(1-h(t))\beta_1(t),\ t\geq 0,\]
   where $\lambda\geq \theta$ is a constant to be determined later.
  Set $V(t,x,i)=\beta_i(t)\rho(t,x)$ for $i=1,2$, $(t,x)\in \R^+\!\times\R^d$.
  Then
  \begin{equation}\label{2.10}
  \begin{split}
    \mathscr A_tV(t,x,1)&=\beta_1(t) L^{(1)}_t\rho(t,x)+\beta_1'(t)\rho(t,x)+ q_{12} (\beta_2(t)-\beta_1(t))\rho(t,x)\\
    &\leq \Big(- \theta\psi(t)+q_{12}\big(\frac{\beta_2(t)}{\beta_1(t)}-1\big)
    +\frac{\beta_1'(t)}{\beta_1(t)}\Big)V(t,x,1)\\
    &=\Big(-\lambda\psi(t)+\frac{\beta_1'(t)}{\beta_1(t)}\Big)V(t,x,1),\quad (t,x)\in \R^+\!\times\R^d,
  \end{split}
  \end{equation}
  \begin{equation}\label{2.11}
  \begin{split}
    \mathscr A_t V(t,x,2)&=\beta_2(t) L^{(2)}_t\rho(t,x)+\beta_2'(t)\rho(t,x)+q_{21}(\beta_1(t)-\beta_2(t))\rho(t,x)\\
    &\leq \Big(-\psi(t)+q_{21}\big(\frac{\beta_1(t)}{\beta_2(t)}-1\big)
    +\frac{\beta_2'(t)}{\beta_2(t)}\Big)V(t,x,2)\\
    &=\Big(-\psi(t)+\frac{q_{21}(\lambda-\theta)\psi(t)}{q_{12}-(\lambda-\theta)\psi(t)}
    +\frac{\beta_2'(t)}{\beta_2(t)}\Big)V(t,x,2), \quad (t,x)\in \R^+\times\R^d.
  \end{split}
  \end{equation}
  Since
  \[\lim_{t\ra \infty}\frac{-\psi(t)+
  \frac{q_{21}(\lambda-\theta)\psi(t)}{q_{12}-(\lambda-\theta)\psi(t)}}{\psi(t)}
  =-1+\kappa(\lambda-\theta),\]
  for any $\veps>0$ there exists a constant $r_0>0$ such that for all $t>r_0$,
  \begin{equation}\label{2.12}
  \frac{-\psi(t)+
  \frac{q_{21}(\lambda-\theta)\psi(t)}{q_{12}-(\lambda-\theta)\psi(t)}}{\psi(t)}
  \leq -1+\kappa(\lambda-\theta)+\veps.
  \end{equation}
  Moreover,
  \begin{equation*}
    \frac{\beta_2'(t)}{\beta_2(t)}=-\frac{(\lambda-\theta)\psi'(t)}
    {q_{12}-(\lambda-\theta)\psi(t)}+\frac{\beta_1'(t)}{\beta_1(t)}.
  \end{equation*}
  As $\psi(t)$ decreases to $0$, there exists a constant $r_1>0$ such that
  for all $t>r_1$, $\lambda\in [\theta,2]$, $q_{12}-(\lambda-\theta)\psi(t)>0$, and hence
  \begin{equation}\label{2.13}
  \frac{\beta_2'(t)}{\beta_2(t)}\geq \frac{\beta_1'(t)}{\beta_1(t)}.
  \end{equation}
  Combining \eqref{2.10}, \eqref{2.11} with \eqref{2.12}, \eqref{2.13}, we get
  \begin{equation}\label{2.14}
  \mathscr A_t V(t,x,i)\leq \Big(\max\{-\lambda, -1+\kappa(\lambda-\theta)+\veps\}\psi(t)+\frac{\beta_2'(t)}{\beta_2(t)}\Big)V(t,x,i),
  \end{equation} for $t\geq r_0\vee r_1:=\max\{r_0,r_1\},\ x\in\R^d,\ i=1,2$.
  Now, by the arbitrariness of $\lambda$ on $[\theta,\infty)$, we obtain   that
  \begin{align*}&\min_{\theta\leq \lambda<\infty}\max\{-\lambda,-1+\kappa(\lambda-\theta)+\veps\}
  \\
  &=\min_{\theta\leq \lambda\leq 2}\max\{-\lambda,-1+\kappa(\lambda-\theta)+\veps\}\\
  &=-\frac{\kappa\theta+1-\veps}{1+\kappa}
  =:-\eta_\veps(\theta,\kappa).
  \end{align*}
  It is clear that $\lim_{\veps\ra 0}\eta_{\veps}(\theta,\kappa)=\eta(\theta,\kappa)$.
  Consequently, there exists $\lambda\geq \theta$ and corresponding function $\beta_2(t)$ and $V(t,x,i)$ such that
  \begin{equation}\label{2.15}
  \mathscr A_tV(t,x,i)\leq \big(-\eta_{\veps}(\theta,\kappa)\psi(t)+\frac{\beta_2'(t)}{\beta_2(t)}\big)V(t,x,i)
  \end{equation} for $t\geq  r_0\vee r_1$, $x\in\R^d$, $i=1,2$.

  We also need to compare the value of $\beta_1(t)$ and $\beta_2(t)$. As $\frac{\beta_2(t)}{\beta_1(t)}=1-h(t)$ and $\lim_{t\ra \infty} h(t)=\lim_{t\ra \infty}(\lambda-\theta)\psi(t)/q_{12}=0$, there exists a constant $r_2>0$ such that for all $t\geq r_2$,
  \[(1-\veps)\beta_1(t)\leq \beta_2(t)\leq \beta_1(t)(1+\veps).\]
  Invoking $\rho(t,x)\geq |x|^p$ for $(t,x)\in \R^+\times\R^d$, we get
  \begin{equation}\label{2.16}
  V(t,x,2)\geq \beta_2(t) |x|^p,\ \ V(t,x,1)\geq\beta_1(t)|x|^p\geq (1+\veps)^{-1}\beta_2(t)|x|^p,
  \end{equation} for $t\geq r_2$ and $x\in \R^d$.

  Set
  \[\zeta(t)=\!\!\max\Big\{\!-\theta\psi(t)+q_{12}\big(\frac{\beta_2(t)}{\beta_1(t)}-1\big),-\theta\psi(t)
  +q_{21}\big(\frac{\beta_1(t)}{\beta_2(t)}-1\big)\Big\}, \ t\in [0,r_0\vee r_1],\]
  and
  \[\zeta(t)=-\eta_{\veps}(\theta,\kappa),\quad t>r_0\vee r_1.\]
  By It\^o's formula and \eqref{2.10}, \eqref{2.11}, \eqref{2.15}, we have
  \[\E V(t,X_t,\La_t)\leq \E V(s,X_s,\La_s)+\E\int_s^t\Big(\zeta(r)\psi(r)+\frac{\beta_2'(r)}{\beta_2(r)}\Big)
  V(r,X_r,\La_r)\,\d r,\quad 0\leq s<t.\]
  Hence, by Gronwall's inequality and \eqref{2.16}, we get
  \[(1+\veps)^{-1}\beta_2(t)\E |X_t|^p\leq V(0,x_0,i_0)\exp\Big(\int_0^t\big(\zeta(s)\psi(s)+\frac{\beta_2'(s)}{\beta_2(s)}\big)\d s\Big).\]
  Invoking that $\lim_{t\ra\infty}\int_0^t\psi(s)\d s=\infty$,
  \begin{equation}\label{2.17}
  \begin{split}
  \limsup_{t\ra \infty} \frac{\log \E|X_t|^p}{\int_0^t\psi(s)\d s}&\leq \limsup_{t\ra \infty} \frac{\int_0^t\zeta(s)\psi(s)\d s}{\int_0^t\psi(s)\d s}\\
  &\leq \limsup_{t\ra \infty}\frac{\int_0^{r_0\vee r_1}\zeta(s)\psi(s)\d s-\int_{r_0\vee r_1}^t \eta_{\veps}(\theta,\kappa)\psi(s)\d s}{\int_0^t\psi(s)\d s}\\
  &=\limsup_{t\ra \infty}\frac{-\int_{r_0\vee r_1}^t\eta_\veps(\theta,\kappa)\psi(s)\d s}{\int_0^t\psi(s)\d s}=-\eta_\veps(\theta,\kappa).
  \end{split}
  \end{equation}
  Note that the left hand side of \eqref{2.17} is independent of $\veps$. So, letting  $\veps$ tend to $0$, we obtain
  \begin{equation*}\label{2.18}
  \limsup_{t\ra \infty} \frac{\log\E|X_t|^p}{\int_0^t\psi(s)\d s}\leq -\eta(\theta,\kappa).
  \end{equation*}
  To complete the proof, we only need to use the similar argument of Lemma \ref{t-sw} to obtain the almost surely asymptotical stability of $(X_t)$, whose details are omitted.
\end{proof}

The questions posed in the introduction can be easily answered by Theorem \ref{prop-2}.
For Example \ref{ex-1}, we take $\rho(t,x)=x^2$, then we can set $\psi(t)\equiv 1$ and $\theta=0$ in \eqref{con-1}. By \eqref{2.9}, it yields
\[\limsup_{t\ra \infty} \frac{\log \E X_t^2}{t}\leq -\frac{q_1}{q_1+q_2}.\]
Similar discussion can yield further the following interesting conclusion. When $(X_t)$ is $p$-th moment stable with decay $\lambda_1(t)$ in the fixed state ``1'' and $p$-th moment stable with decay $\lambda_2(t)$ in the fixed state ``2" satisfying $\lim_{t\ra \infty} \lambda_2(t)/\lambda_1(t)=0$. Then $(X_t)$ will be ultimately $p$-th moment stable with decay $\lambda_1(t)$.
For Example \ref{ex-2}, we still take $\rho(t,x)=x^2$, then we set $\psi(t)=\frac{m}{1+t}$ and take $\theta=1/m$ in \eqref{con-1}. By \eqref{2.9}, we get
\[\limsup_{t\ra \infty} \frac{\log \E X_t^2}{\log (1+t)^m}\leq - \frac{q_1+q_2/m}{q_1+q_2}.\]

Next, we extend Theorem \ref{prop-2} to deal with  the case $N>2$.
\begin{prop}\label{prop-3}
Let $(X_t,\La_t)$ be a regime-switching diffusion process satisfying \eqref{1.1} with $\S=\{1,\ldots,N\}$ and $N\geq 2$. Let $\psi\in C^1(\R^+)$ be a positive decreasing function satisfying $\dis\lim_{t\ra \infty}\psi(t)=0$ and $ \lim_{t\ra \infty}\int_0^t\psi(s)\d s=\infty$. Assume that there exists a function $\rho\in C^{1,2}(\R^+\times\R^d;\R^+)$, positive constants $p$, $c_1,\,\ldots,\,c_N$ such that
$\rho(t,x)\geq |x|^p$ and
\[L^{(i)}\rho(t,x)\leq -c_i\rho(t,x)\]
for all $(t,x)\in \R^+\times\R^d$, $i\in \S$. Set
\begin{align*}
  \S_0=\big\{i\in \S;\ c_i=\min\{c_k;k\in\S\}\big\},\quad \S_1=\S\backslash\S_0.
\end{align*}
Set $q_{jS_1}=\sum_{i\in \S_1} q_{ji}$ for $j\in \S_0$ and $q_{i\S_0}=\sum_{j\in\S_0}q_{ij}$ for $i\in \S_1$.
Assume that both $\S_0$ and $\S_1$ are not empty, and $q_{jS_1}>0$ for every $j\in \S_0$. Set \[\kappa=\frac{\max_{i\in\S_1} q_{i\S_0}}{\min_{j\in \S_0}q_{jS_1}},\  \ \text{and}\ \ \eta=\frac{\tilde c_2+\kappa \tilde c_1}{1+\kappa},\] where $\tilde c_1=\min\{c_i; i\in\S_0\}$, $\tilde c_2=\min\{c_i; i\in \S_1\}$.
Then
\begin{equation}\label{2.19}
\limsup_{t\ra \infty}\frac{\log \E|X_t|^p}{\int_0^t \psi(s)\d s}\leq -\eta<0.
\end{equation}
\end{prop}

\begin{proof}
  For the sake of simplicity of notation, we may reorder $\S$ so that
  \[c_1\leq c_2\leq \ldots\leq c_N.\]
  There exists some $k_0\leq N$ such that $c_1=\ldots=c_{k_0}$ and hence $\S_0=\{1,\ldots,k_0\}$. If $k_0=N$, then the desired result is trivial. So, we only need to consider the case $k_0<N$.

  Let $\beta_1(t)\in C^1(\R^+)$ be a positive function and $\beta_1(t)\geq c>0$ for all $t\geq 0$.
  Set \[h(t)=\frac{(\lambda-c_1)\psi(t)}{\min_{j\in\S_0}q_{jS_1}}, \quad \beta_2(t)=(1-h(t))\beta_1(t),\] which is well-defined due to our assumption that $q_{jS_1}>0$ for every $j\in\S_0$. As $\lim_{t\ra \infty} \psi(t)=0$, $\beta_2(t)<\beta_1(t)$ for $t$ large enough.
  Define $V(t,x,i)=\rho(t,x)\beta_1(t)$ for $i\in\S_0$, and $V(t,x,i)=\rho(t,x)\beta_2(t)$ for $i\in\S_1$.
  Then, we obtain that for $t$ large enough, for $i\in \S_0$,
  \begin{align*}
    \mathscr A_t V(t,x,i)&\leq \Big[-c_1\psi(t)+q_{iS_1}\Big(\frac{\beta_2(t)}{\beta_1(t)}-1\Big)
    +\frac{\beta_1'(t)}{\beta_1(t)}\Big]V(t,x,i)\\
    &\leq \Big[-\lambda \psi(t)+\frac{\beta_1'(t)}{\beta_1(t)}\Big]V(t,x,i),
  \end{align*}
  and for $i\in \S_1$,
  \begin{align*}
    \mathscr A_t V(t,x,i)&\leq
    \Big[-c_i\psi(t)+q_{i\S_0}\Big(\frac{\beta_1(t)}{\beta_2(t)}-1\Big)
    +\frac{\beta_2'(t)}{\beta_2(t)}\Big]V(t,x,i)\\
    &\leq \Big[-\tilde c_1\psi(t)+\big(\max_{i\in\S_1}q_{i\S_0}\big)\Big(\frac{h(t)}{1-h(t)}\Big)
    +\frac{\beta_2'(t)}{\beta_2(t)}\Big]V(t,x,i)\\
    &\leq \Big[-\tilde c_1\psi(t)+\big(\max_{i\in\S_1}q_{i\S_0}\big)
    \frac{(\lambda-c_1)\psi(t)}{\min_{j\in\S_0}q_{j\S_1}-(\lambda-c_1)\psi(t)}
    +\frac{\beta_2'(t)}{\beta_2(t)}\Big]V(t,x,i).
  \end{align*}
Then, applying  a similar technique used in the argument of Theorem \ref{prop-2}, we can obtain the desired result \eqref{2.19}.
\end{proof}

Based on the theory of M-matrix, we can provide an interesting result on the $p$-th moment stability with exponential decay (Theorem \ref{cor-1} below). Recall first some notation on M-matrix. By $B\geq 0$ for a square matrix $B$, we mean that all elements of $B$ is nonnegative. A square matrix $A=(a_{ij})_{n\times n}$ is called an    M-Matrix  if $A$ can be expressed in the form $A=sI-B$ with some $B\geq 0$ and $s\geq\mathrm{Ria}(B)$, where $I $ denotes the $n\times n$ identity matrix, and $\mathrm{Ria}(B)$ denotes the spectral radius of $B$. When $s>\mathrm{Ria}(B)$,   $A$ is called a  nonsingular M-matrix. There exist many equivalent definitions of the nonsingular M-matrix. We refer the reader to \cite{BP} or \cite{SX14} for the characterization of nonsingular M-matrix.

\begin{lem}\label{prop-1}
  Let $(X_t,\La_t)$ be the same process as that of Lemma \ref{t-sw}. Assume that there exist a function $\rho\in C^{1,2}(\R^+\times\R^d;\R^+)$,   constants $p>0$ and $c_i\in \R$ for every $i\in\S$ such that
  \begin{itemize}
    \item[(i)]\ $\rho(t,x)\geq |x|^p \quad \forall\,(t,x)\in \R^+\times \R^d$;
    \item[(ii)]\ $L_t^{(i)} \rho(t,x)\leq c_i\,\rho(t,x)\quad \forall\,(t,x)\in \R^+\times \R^d$.
  \end{itemize}
  $\diag(c_1,\ldots,c_N)$ denotes the diagonal matrix generated by the vector $(c_1,\ldots,c_N)$ as usual. If the matrix $-\big(\diag(c_1,\ldots,c_N)+Q\big)$ is a nonsingular M-matrix, then there exists a constant $\gamma>0$ such that
  \[\limsup_{t\ra\infty}\frac{\log \E|X_t|^p}{t}\leq -\gamma<0,\]
  which means that $(X_t)$ is $p$-th moment stable with exponential decay.
\end{lem}

\begin{proof}
  As the matrix $-(\diag(c_1,\ldots, c_N) +Q)$ is a nonsingular M-matrix, there exists a vector $\bm{\xi}\gg 0$ such that
  \[\bm{\beta}=(\beta_1,\ldots,\beta_N)^\ast=(\diag(c_1,\ldots, c_N)+Q)\bm{\xi}\ll 0,\]
  i.e. all elements of $\bm{\beta}$ are strictly less than $0$.
  Set $V(t,x,i)=\rho(t,x)\xi_i$ for $(t,x,i)\in \R^+\!\times\R^d\times\! \S$. By virtue of  the conditions (ii), we get
  \begin{align*}
    \mathscr A_t V(t,x,i)
    &\leq \big(\diag(c_1,\ldots, c_N)+Q\big)\bm{\xi}(i)\rho(t,x)\\
    &\leq \max_{k\in \S}\beta_k \,\rho(t,x)\leq \frac{\max_{k\in \S}\beta_k}{\min_{k\in \S}\xi_k}V(t,x,i).
  \end{align*}
  According to Lemma \ref{t-sw}, we obtain that
  \[\limsup_{t\ra \infty}\frac{\E|X_t|^p}{t}\leq \frac{\max_{k\in \S}\beta_k}{\min_{k\in \S}\xi_k}=:-\gamma<0.\]
\end{proof}

\begin{thm}\label{cor-1}
Let $(X_t,\La_t)$ be a regime-switching diffusion process satisfying \eqref{1.1}.  Assume that there exist a function $\rho \in C^{1,2}(\R^+\times\R^d;\R^+)$, a state $i_0\in\S$, constants $c_{i_0}>0$, $p>0$ and continuous nonnegative functions $\psi_{i}$ for $i\in \S\backslash\{i_0\}$ such that
\begin{itemize}
    \item[(i)]\ $\rho(t,x)\geq |x|^p \quad \forall\,(t,x)\in \R^+\times \R^d$;
    \item[(ii)]\ $L^{(i_0)}\rho(t,x)\leq -c_{i_0} \rho(t,x)\quad \forall\, (t,x)\in\R^+\times \R^d$;
    \item[(iii)]\ $L_t^{(i)} \rho(t,x)\leq -\psi_i(t)\rho(t,x)\quad \forall\,(t,x)\in \R^+\times \R^d$, $i\in \S\backslash\{i_0\}$.
  \end{itemize}
Then there exists a constant $\gamma>0$ such that
\begin{equation}\label{e-7}\limsup_{t\ra \infty}\frac{\log\E|X_t|^p}{t}\leq -\gamma<0.\end{equation} Moreover, the process $(X_t)$ is also almost surely asymptotically stable.
\end{thm}

\begin{proof} Condition (iii) yields that
\[L_t^{(i)}\rho(t,x)\leq 0,\quad \forall\, (t,x)\in \R^+\!\times\R^d,\ i\in \S\backslash\{i_0\}.\]
Set $\Theta_{i_0}=(\theta_{kl})$ is an $N\times N$ matrix with $\theta_{i_0i_0}=-c_{i_0}$ and $\theta_{kl}=0$ for $(k,l)\neq (i_0,i_0)$. It is known that $\big(\Theta_{i_0}+Q\big)$ associates with a continuous time Markov chain with killing. Therefore, by virtue of \cite[Theorem 2.27]{Dell}, $-\big(\Theta_{i_0}+Q\big)$ is a nonsingular M-matrix if $\big(\Theta_{i_0}+Q\big)$ is nonsingular. Here, one should note that the definition of M-matrix in
\cite[Definition 2.1]{Dell} coincides with the definition of nonsingular M-matrix in this work due to the assertion \cite[$\textrm{N}_{38}$, pp.137]{BP}. As $-\big(\Theta_{i_0}+Q\big)$ is an irreducibly diagonally dominant matrix, it follows from the L\'evy-Desplanques theorem that $-\big(\Theta_{i_0}+Q\big)$ is nonsingular. Therefore, according to Lemma \ref{prop-1}, there exists a constant $\gamma>0$ such that
\[\limsup_{t\ra\infty}\frac{\log\E |X_t|^p}{t}\leq -\gamma<0,\]
which means that $(X_t)$ is $p$-th moment exponentially stable. The almost surely asymptotical stability of $(X_t)$ follows immediately from Lemma \ref{t-sw}.
\end{proof}

%
%

\begin{rem}
  The condition (ii) of Theorem \ref{cor-1} means that the diffusion corresponding to $(X_t)$ in the fixed state $i_0$ is $p$-th moment exponentially stable. The condition (iii) means that the $p$-th moment of the diffusions corresponding to $(X_t)$ in other fixed states are at least $p$-th moment stable with algebraic decay if $\lim_{t\ra \infty}\int_0^t\psi_i(s)\d s=\infty$ for $i\in\S\backslash\{i_0\}$. According to Theorem \ref{cor-1},  $(X_t)$  is $p$-th moment exponentially stable regardless of the $Q$-matrix of $(\La_t)$ when $Q$ is irreducible. For Example \ref{ex-1} in the introduction, the diffusion process corresponding to $(X_t)$ in the fixed state ``1" is mean square exponentially stable and is stable with polynomial decay in the fixed state ``2". Hence, due to Theorem \ref{cor-1}, we can also prove that $(X_t)$ is mean square exponentially stable.
\end{rem}

In the end, we apply our results on the stability of regime-switching diffusion processes to the stabilization problem by a feedback control. Consider the following SDE:
\begin{equation}\label{s-1}\d X_t=b(t,X_t)\d t+\sigma(t,X_t)\d B_t.
\end{equation}
If this given SDE is not stable, it is traditional to design a feedback control $u(t,X_t)$ in order for the controlled system
\begin{equation}\label{s-2}\d X_t=(b(t,X_t)+u(t,X_t))\d t +\sigma(t,X_t)\d B_t\end{equation}
to become stable. Such a regular feedback control requires the continuous observations of the state $X_t$ for all $t\geq 0$. This is of course expensive. To be costless and more realistic, X. Mao and his collaborators in \cite{Mao13, Mao15} investigated the stabilization of a given hybrid SDE by discrete-time feedback control. Based on their method, one can design a feedback control $u(t,X([t/\tau]\tau))$ based on the discrete-time observations of $X_t$ at times $\tau,\,2\tau,\ldots$ so that the controlled system
\begin{equation}\label{s-3}\d X_t=(b(t,X_t)+u(t, X([t/\tau]\tau)))\d t+\sigma(t,X_t)\d B_t\end{equation}
becomes stable. The results in  \cite{Mao13, Mao15} require the step size $\tau$ to be small enough to guarantee the discrete-time controlled system being stable via studying  the stability of the  corresponding continuous time controlled system. Therefore, this method eliminates the observation times of $X_t$, but one still needs a large number of observations as $\tau $ is small.

Applying our results on the stability of regime-switching diffusion processes, we can design a new approach to stabilize the SDE \eqref{s-1}. We first realize a continuous time Markov chain $(\La_t)$ on the state space $\S=\{1,2\}$ independent of the Brownian motion in SDE \eqref{s-1}, then construct a feedback control $u(t,X_t,\La_t)$ such that $u(t,x,1)\equiv 0$ and the controlled system
\begin{equation}\label{s-4}
\d X_t=(b(t,X_t)+u(t,X_t,\La_t))\d t+\sigma(t,X_t)\d B_t
\end{equation}
to be stable. For instance, $u(t,X_t,\La_t)$ is designed so that we can use the criteria established in  Theorems \ref{prop-1} and \ref{cor-1}
to check the stability of SDE \eqref{s-4}. If the Markov chain $(\La_t)$ starts at $\La_0=1$, then it will stay at $``1"$ for a random time $T_1$, which is exponentially distributed with rate parameter $q_1$, i.e. $\p(T_1>t)=e^{-q_1 t}$. After that the process $(\La_t)$ jumps to $``2"$ and sojourns at $``2"$ for a random time $T_2$, which is exponentially distributed with parameter $q_2$. Then $(\La_t)$ jumps back to $``1"$ and so on. Let $\alpha(t)$ denote the time spent by $(\La_t)$ at the state $``1"$ during the time interval $[0,t]$. It is known that (cf. for example, \cite{Tak}) \begin{equation}\label{soj}
\lim_{t\ra\infty} \frac{\E [\alpha(t)]}{t}=\frac{q_2}{q_1+q_2}.
\end{equation} For example, when $q_1=q_2$, the equality \eqref{soj} implies that the process $(\La_t)$ spends half of time to stay at the state $``1"$ in average for large $t$. Since we do not need to observe $X_t$ when the Markov chain $(\La_t)$ is at state $``1"$, this method can
greatly reduce the number of observations of $X_t$ in average. Moreover, one can combine this kind of feedback control with X. Mao's results in \cite{Mao13} to further reduce the number of observations. This is a kind of feedback control putting in a random environment.

\end{document}